\documentclass{article}
\usepackage[utf8]{inputenc}
\usepackage[left=1in,top=1in,right=1in,bottom=1in]{geometry}
\usepackage[linktocpage=true]{hyperref}
\usepackage{setspace}
\usepackage{amssymb, amsmath, amsthm, graphicx,mathrsfs}
\usepackage{caption,cite}
\usepackage{mathtools,color, xcolor}

\usepackage{cleveref}
 \usepackage{enumitem}

\newtheorem{thm}{Theorem}[section]

\newtheorem{defn}[thm]{Definition}

\newtheorem{lem}[thm]{Lemma}
\newtheorem{conj}{Conjecture}
\newtheorem{prop}[thm]{Proposition}
\newtheorem{claim}[thm]{Claim}

\newcommand{\C}{\mathcal{C}}
\newcommand{\F}{\mathbb{F}}

\newcommand{\E}{\mathbb{E}}

\renewcommand{\P}{\mathrm{Pr}}
\renewcommand{\l}{\left}
\renewcommand{\r}{\right}

\title{Off-diagonal Ramsey numbers  for slowly growing hypergraphs}
\author{Sam Mattheus\footnote{Department of Mathematics, Vrije Universiteit Brussel, Pleinlaan 2, 1050 Brussel, Belgium. {\tt sam.mattheus@vub.be}. Research supported by a postdoctoral fellowship 1267923N from the Research Foundation Flanders (FWO).} \and Dhruv Mubayi\footnote{Department of Mathematics, Statistics and Computer Science, University of Illinois Chicago, Chicago, IL 60607, USA. {\tt mubayi@uic.edu}. Research supported by NSF grants DMS-1952767, DMS-2153576 and a Simons Fellowship.} \and Jiaxi Nie\footnote{School of Mathematics, Georgia Institute of Technology,
Atlanta, GA 30332 USA. {\tt jnie47@gatech.edu}.} \and Jacques Verstra\"{e}te \footnote{Department of Mathematics, University of California, San Diego, La Jolla, CA 92093, USA. {\tt jacques@ucsd.edu}. Research supported by NSF grant DMS-1952786.}}
\date{\today}

\begin{document}

\maketitle


\begin{abstract}
For a $k$-uniform hypergraph $F$ and a positive integer $n$, the Ramsey number $r(F,n)$ denotes the minimum $N$ such that every $N$-vertex $F$-free $k$-uniform hypergraph contains an independent set of $n$ vertices. A  hypergraph is {\em slowly growing} if there is an ordering $e_1,e_2,\dots,e_t$ of its edges such that  
$|e_i \setminus \bigcup_{j = 1}^{i - 1}e_j| \leq  1$ for each $i \in \{2, \ldots, t\}$. We prove that if $k \geq 3$ is fixed and $F$ is any non $k$-partite slowly growing $k$-uniform hypergraph, then for $n\ge2$,   
\[ r(F,n) = \Omega\Bigl(\frac{n^k}{(\log n)^{2k - 2}}\Bigr).\]
In particular, we deduce that  the off-diagonal Ramsey number $r(F_5,n)$ is of order $n^{3}/\mbox{polylog}(n)$, where $F_5$ is the triple system $\{123, 124, 345\}$. This is the only 3-uniform Berge triangle for which the  polynomial power of its off-diagonal Ramsey number was not previously known.  
Our constructions use pseudorandom graphs,  martingales, and hypergraph containers.
\end{abstract}

\bigskip

\section{Introduction}
A hypergraph is a pair $(V,E)$ where $V$ is a set, whose elements are called vertices, and $E$ is a family of nonempty subsets of $V$, whose elements are called edges. A $k$-uniform hypergraph ($k$-graph for short) is a hypergraph whose edges are all of size $k$. An {\em independent set} of a hypergraph $F$ is a subset of $V(F)$ which does not contain any edge of $F$.

Given a $k$-graph $F$, the {\em off-diagonal Ramsey number} $r(F,n)$ is the minimum integer such that every $F$-free $k$-graph on $r(F,n)$ vertices has an independent set of size $n$.  Ajtai, Koml\'os and Szemer\'edi~\cite{AKS1980}  proved the upper bound $r(K_3,n)=O(n^2/\log n)$, and Kim~\cite{Kim1995} proved the corresponding lower bound $r(K_3,n)=\Omega(n^2/\log n)$. The current state-of-the-art results are due to Fiz Pontiveros, Griffiths and Morris~\cite{fiz2020triangle} and Bohman and Keevash~\cite{bohman2021dynamic}, who determine $r(K_3,n)$ up to a small constant factor:
$$
\l(\frac{1}{4}-o(1)\r)\frac{n^2}{\log n}\le r(K_3,n)\le \l(1+o(1)\r)\frac{n^2}{\log n}.
$$

For larger cliques, the current best general lower bounds are obtained by Bohman and Keevash~\cite{bohman2010early} strengthening earlier bounds of Spencer~\cite{spencer1975ramsey,spencer1977asymptotic}. On the other hand, the current best upper bounds are proved by Li, Rousseau and Zang~\cite{li2001asymptotic} by extending ideas of Shearer~\cite{shearer1983note}, which improve earlier bounds of Ajtai, Komlós and Szemerédi~\cite{AKS1980}. These bounds are as follows: for $s\ge 3$, there exists constant $c_1(s)>0$ such that
$$
c_1(s)\frac{n^{\frac{s+1}{2}}}{(\log n)^{\frac{s+1}{2}-\frac{1}{s-2}}}\le r(K_s,n)\le (1+o(1))\frac{n^{s-1}}{(\log n)^{s-2}}.
$$

Recently, the first and fourth authors~\cite{mattheus2024asymptotics} determined the asymptotics of $r(K_4,n)$ up to a logarithmic factor by proving the following lower bounds.
\begin{thm}[Theorem 1, \cite{mattheus2024asymptotics}]
As $n\rightarrow\infty$,
$$
r(K_4,n)=\Omega\l(\frac{n^3}{(\log n)^4}\r).
$$
\end{thm}

In this paper, we prove some hypergraph versions of these results. A {\em Berge triangle} is a hypergraph consisting of three edges $e_1,~e_2$ and $e_3$ such that there exists three vertices $x,~y$ and $z$ with the property that $\{x,y\}\subset e_1$, $\{y,z\}\subset e_2$, and $\{x,z\}\subset e_3$. It is easy to check that there are only four different 3-uniform Berge triangles: $LC_3$ (loose cycle of length 3), $TP_3$ (tight path on three edges and five vertices), $F_5$, and $K^{3-}_4$ (3-uniform clique on four vertices minus an edge), as shown from left to right in Figure~\ref{fig:BergeTriangle}. It is natural to consider the problem of determining the off-diagonal Ramsey numbers for 3-uniform Berge triangles since they are in some sense the smallest non-trivial hypergraphs that provide a natural extension of $r(K_3,n)$.

\begin{figure}[h]
    \centering
    \includegraphics[scale=0.4]{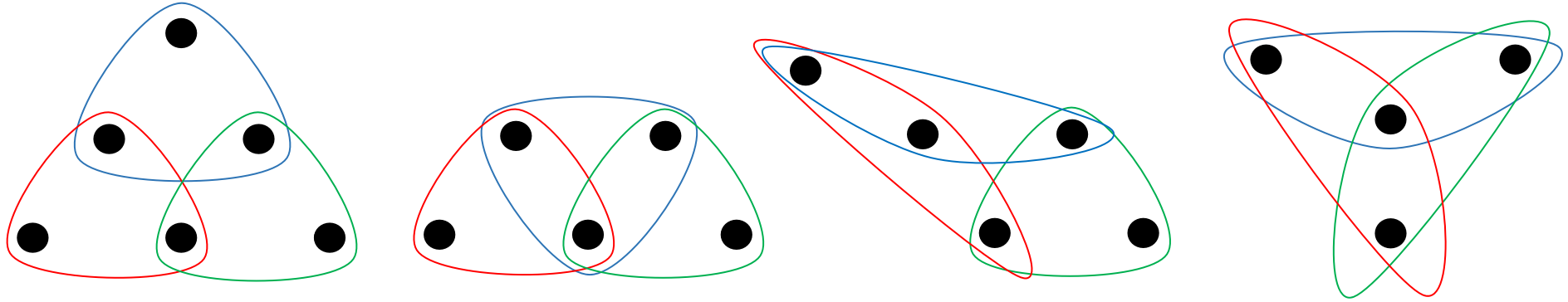}
    \caption{From left to right: $LC_3$, $TP_3$, $F_5$ and $K^{3-}_4$.}
    \label{fig:BergeTriangle}
\end{figure}

The off-diagonal Ramsey numbers for $TP_3$ and $LC_3$ have been determined up to a logarithmic factor: for $TP_3$, a result of Phelps and Rödl~\cite{phelps1986steiner} shows that $c_1n^2/\log n\le r(TP_3,n)\le c_2n^2$; for $LC_3$, Kostochka, the second author and the fourth author~\cite{kostochka2013hypergraph} showed that $c_1n^{3/2}/(\log n)^{3/4}\le r(LC_3,n)\le c_2 n^{3/2}$. It seems plausible to conjecture that for some constant $c$,
\[ r(TP_3,n) \leq \frac{cn^2}{\log n} \quad \mbox{ and } \quad r(LC_3,n) \leq \frac{cn^{\frac{3}{2}}}{(\log n)^{\frac{3}{4}}}.\]
It is conjectured explicitly in~\cite{kostochka2013hypergraph} that $r(LC_3,n) = o(n^{3/2})$ and the question of determining the order of magnitude of $r(TP_3,n)$ was posed in~\cite{CM}. It was also shown in~\cite{CM} that $r(TP_4, n)$ has order of magnitude $n^2$, leaving $TP_3$ as the only tight path for which the order of magnitude of $r(TP_s, n)$ remains open. We remark that if one can prove that every $n$-vertex $TP_3$-free 3-graph with average degree $d>1$  has an independent set of size at least $\Omega(n \sqrt {\log d / d})$, then  this implies that $r(TP_3, n) = \Theta(n^2/\log n)$.

The problem for $K^{3-}_4$ is interesting in the sense that it is the smallest hypergraph whose off-diagonal Ramsey number is at least exponential: Erd\H os and Hajnal~\cite{EH72} proved $r(K^{3-}_4, n) = n^{O(n)}$ and R\"odl (unpublished) proved $r(K^{3-}_4,n)\ge 2^{\Omega(n)}$. More recently,  Fox and He~\cite{fox2021independent} showed that $r(K^{3-}_4,n)=n^{\Theta(n)}$. 

The problem for $F_5$, however, is not very well studied: a result of Kostochka, the second author and the fourth author~\cite{kostochka2014independent} implies that $r(F_5,n)\le c_1n^3/\log n$, and the standard probabilistic deletion method shows that $r(F_5,n)\ge c_2n^{2}/\log n$. In this paper, we fill this gap by showing that $r(F_5,n)=n^{3}/\mbox{polylog}(n)$. This is a consequence of a more general theorem that we will prove.

Building upon techniques in~\cite{mattheus2024asymptotics}, we prove lower bounds for the off-diagonal Ramsey numbers of a large family of hypergraphs. A $k$-graph $F$ is {\em slowly growing} if its edges can be ordered as $e_1,\dots,e_t$ such that
$$
\forall \, i \in \{2,\dots,t\},~\Bigl|e_i\setminus \bigcup_{j = 1}^{i - 1} \, e_j\Bigr|\le 1.
$$
We use this terminology to describe the fact that at most one new vertex is added when we add a new edge in the ordering. Further, $F$ is $k$-partite, or {\em degenerate}, if its vertices can be partitioned into $k$ sets $V_1,\dots, V_k$ such that each edge intersects each $V_i$, $1\le i\le k$, in exactly one vertex. Otherwise, $H$ is {\em non-degenerate}. The three hypergraphs $TP_3$, $F_5$ and $K_4^{3-}$ in Figure~\ref{fig:BergeTriangle} are slowly growing, whereas the first is not. The last two are non-degenerate.

In this paper, we obtain the following result for non-degenerate slowly growing hypergraphs.
\begin{thm}\label{theorem:main}
For every $k\ge 3$, there exists a constant $c>0$ such that for every slowly growing non-degenerate $k$-graph $F$ and all integers $n\ge 2$
$$
r(F,n)\ge \frac{cn^k}{(\log n)^{2k-2}}.
$$
\end{thm}

\Cref{theorem:main} is tight up to a logarithmic factor for the following family of hypergraphs which includes $F_5$.
For $k\ge 3$, let $F_{2k-1}$ be the $k$-graph on $2k-1$ vertices $v_1,\dots,~v_{k-1},~w_1,~\dots,~w_k$ with $k$ edges $\{v_1,~\dots,~v_{k-1},~w_i\}$, $1\le i\le k-1$, and $\{w_1,~\dots,~w_k\}$. Further, let $T_k$ be the $k$-graph obtained from $F_{2k-1}$ by adding the edge $\{v_1,\dots,v_{k-1}, w_k\}$. See Figure~\ref{fig:F7T4} for an illustration of $F_7$ and $T_4$. Note that $T_2$ is a (graph) triangle.

\begin{figure}[h]
    \centering
    \includegraphics[scale=0.4]{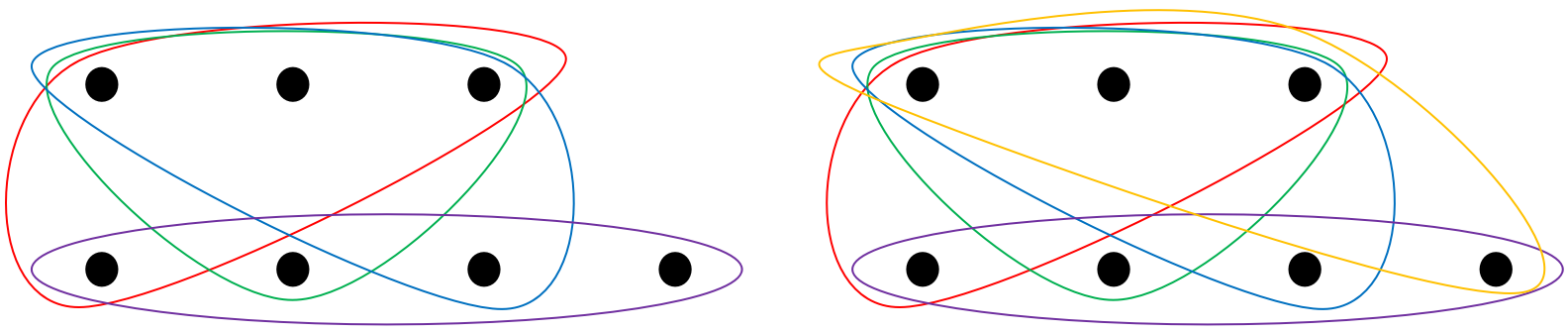}
    \caption{$F_7$ and $T_4$}
    \label{fig:F7T4}
\end{figure}

The order of magnitude of $r(T_k,n)$ for $k \geq 3$ is determined by the upper bound result of Kostochka, the second author and the fourth author~\cite{kostochka2014independent} together with the lower bound result of Bohman, the second author and Picollelli~\cite{bohman2016independent}. For $k = 2$, this theorem restates the known result~\cite{AKS1980,Kim1995,fiz2020triangle,bohman2021dynamic} that $r(K_3,n)$ has order of magnitude $n^2/\log n$.

\begin{thm}[Theorem 2, \cite{kostochka2014independent}; Theorem 1, \cite{bohman2016independent}]\label{theorem:KMV}
Let $k\ge 2$. Then there exist constants $c_1,~c_2>0$ such that for all integers $n\ge 2$,
$$
\frac{c_1n^k}{\log n}\le r(T_k,n)\le \frac{c_2n^k}{\log n}.
$$
\end{thm}

Thus we have $r(F_{2k-1},n)\le r(T_k,n)\le O({n^k}/{\log n})$. On the other hand, it is easy to check that $F_{2k-1}$ is a slowly growing non-degenerate $k$-graph. Hence \Cref{theorem:main} together with \Cref{theorem:KMV} implies the following theorem.

\begin{thm}\label{theorem:F}
Let $k\ge 3$. There exist constants $c_1,c_2>0$ such that for all integers $n\ge 2$,
$$
\frac{c_1n^k}{(\log n)^{2k-2}}\le r(F_{2k-1},n)\le \frac{c_2n^k}{\log n}.
$$
\end{thm}

\Cref{theorem:F} determines $r(F_{2k-1},n)$ up to a logarithmic factor. In particular, this determines $r(F_5,n)$ up to a polylogarithmic factor, and  $F_5$ is the only 3-uniform Berge triangle for which the polynomial power of the off-diagonal Ramsey number was not previously known.

It would be interesting to determine its order of magnitude. We believe the current upper bounds are closer to the truth:

\begin{conj}
There exists a constant $c > 0$ such that for $n \geq 2$, 
\[ r(F_5,n) \geq \frac{cn^3}{\log n}.\]
    \end{conj}

\section{The Construction}
The proof of Theorem~\ref{theorem:main} uses the so-called random block construction, which first requires a pseudorandom bipartite graph. We build our construction using the following bipartite graph.

\begin{defn}
For every prime power $q$ and integer $m\ge 3$, let $\Gamma_{q,m}$ be the bipartite graph with two parts $X=\F^2_q$ and $Y=\F^m_q$, where two vertices $x=(x_0,x_1)\in X$ and $y=(y_0,\dots,y_{m-1})\in Y$ form an edge if and only if
$$
x_1=\sum_{i=0}^{m-1}y_ix_0^{i}
$$
\end{defn}

For any vertices $x,y$ of a graph $G$, we use $d(x)$ to denote the degree of $x$, that is, the number of neighbors of $x$ in $G$, and we use $d(x,y)$ to denote the number of common neighbors of $x$ and $y$. 
It is easy to check the following proposition.

\begin{prop}\label{proposition:Gamma}
For every prime power $q$ and integer $m\ge 2$, $\Gamma_{q,m}$ has the following properties:
\begin{itemize}[align=left]
    \item[{\rm (i)}] $\forall\, x\in X$, $d(x)=q^{m-1}$.
    \item[{\rm (ii)}] $\forall\, y\in Y$, $d(y)=q$.
    \item[{\rm (iii)}] $\forall\, y_1,y_2\in Y$, $d(y_1,y_2)\le m-1$.
\end{itemize}
\end{prop}

For all $k\ge 3$, let $H_{q,k}$ be a $k$-uniform hypergraph on $X=X(\Gamma_{q,k-1})$ whose edges are all $k$-sets $\{x_1,\dots, x_k\}\subseteq X$ such that there exists an element $y\in Y=Y(\Gamma_{q,k-1})$ such that  $\{x_1,\dots, x_k\}\subseteq N(y)$. By \Cref{proposition:Gamma}, $H_{q,k}$ is the union of $q^{k-1}$ 
$k$-uniform cliques on $q$ vertices such that each vertex is contained in $q^{k-2}$ cliques and the vertex sets of every two cliques intersect in at most $k-2$ vertices. Let $H^*_{q,k}$ be the $k$-uniform hypergraph obtained by replacing each maximal clique of $H_{q,k}$ with a random complete $k$-partite $k$-graph on the same vertex set. More formally, for each $y\in Y$, we color the vertices in $N(y)$ with $k$ colors $\{1,\dots,k\}$ uniformly at random, and for each $1\le i\le k$ we let $X_{y,i}\subseteq N(y)$ be the set of vertices with color $i$, and then we replace the clique on $N(y)$ with a complete $k$-partite $k$-graph on $N(y)$ with $k$-partition $X_{y,1}\sqcup\dots\sqcup X_{y,k}$. It is easy to check the following proposition.

\begin{prop}\label{proposition:F-free}
If $F$ is a non-degenerate slowly growing $k$-graph, then $H^*_{q,k}$ is $F$-free.
\end{prop}

\begin{proof}
    Consider an ordering $e_1,\dots,e_t$ of the edges of $F$ such that 
    $$
\forall \, i \in \{2,\dots,t\},~\Bigl|e_i\cap \bigcup_{j = 1}^{i - 1} \, e_j\Bigr|\ge k-1.
$$
    We claim that every copy of $F$ in $H_{q,k}$ must be fully contained in one of the $q^{k-1}$ $k$-uniform cliques of size $q$. Indeed, suppose that we want to build a copy of $F$ in $H_{q,k}$ by consecutively picking the edges in the order given above. Then the fact that every two cliques of $H_{q,k}$ intersect in at most $k-2$ vertices shows that we must pick every  edge in the clique containing the previous edges. Since $H_{q,k}^*$ is obtained from $H_{q,k}$ by replacing every clique by a complete $k$-partite $k$-graph and $F$ itself is not $k$-partite, this proves the statement.
\end{proof}

We will fix an instance of $H^*_{q,k}$ with good {\em Balanced Supersaturation}, which means that each induced subgraph of $H^*_{q,k}$ on $q^{1+o(1)}$ vertices contains many edges that are evenly distributed. Using Balanced Supersaturation together with the Hypergraph Container Lemma ~\cite{balogh2015independent, saxton2015hypergraph}, we can find upper bounds on the number of independent sets in $H^*_{q,k}$ of size $t=(\log q)^2q^{\frac{1}{k-1}}$.

We then take a random subset $W$ of $V(H^*_{q,k})$ where each vertex is sampled independently with probability $p=\Theta(\frac{t}{q})$ as in~\cite{MV2024}. Finally, our construction is obtained by arbitrarily deleting a vertex from each independent set of size $t$ in $H^*_{q,k}[W]$.

\section{Pseudorandomness of $\Gamma_{q,k-1}$}
In this section we show the pseudorandomness of $\Gamma_{q,k-1}$, which will be useful later in showing balanced supersaturation of $H^*_{q,k}$.

Given an $n$-vertex graph $G$, let $A_G$ be the adjacency matrix of $G$, which is the $n \times n$ symmetric matrix where
    \begin{equation*}
        A_G(i,j) := \begin{cases}
               1, & \mbox{if } \{i,j\}\in E(G), \\
               0, & {\rm otherwise.}
            \end{cases}
    \end{equation*} 

Let $\lambda_{1}(G) \geq \dots \geq \lambda_{n}(G)$ denote the eigenvalues of $A_G$. If $G$ is a bipartite graph with bipartition $V_1\sqcup V_2$, we say $G$ is {\em $(d_1,d_2)$-regular} if $d(v)=d_1$ for all $v\in V_1$ and $d(v)=d_2$ for all $v\in V_2$. 

The seminal expander mixing lemma is an important tool that relates edge distribution to the second eigenvalue of a graph. Here we make use of the bipartite version. 

\begin{lem}[Theorem 5.1,~\cite{Hae95}]\label{lemma:expander mixing}
Suppose that $G$ is a $(d_1, d_2)$-regular bipartite graph with bipartition $V_1\sqcup V_2$. Then for every $S\subset V_1$ and $T\subset V_2$, the number of edges between $S$ and $T$, denoted by $e(S, T)$, satisfies 
$$
        \left| e(S,T) - \frac{d_2}{|V_1|}|S||T| \right|
        \le \lambda_{2}(G)\sqrt{|S||T|}.
$$
\end{lem}

By \Cref{proposition:Gamma}, we know $\Gamma_{q,k-1}$ is $(q^{k-2},q)$-regular. For convenience, from now on we let $n=|V(\Gamma_{q,k-1})|=q^2+q^{k-1}$, $A=A_{\Gamma_{q,k-1}}$, $\lambda_i=\lambda_i(\Gamma_{q,k-1})$ for all $1\le i\le n$, and let $d_1=q^{k-2}$, $d_2=q$. 

\begin{lem}\label{lemma:incidence_eigenvalue}
$\lambda_2=q^{\frac{k}{2}-1}$.
\end{lem}

\begin{proof}
Define the matrix
$$
M=\left[ 
        \begin{array}{c|c} 
          0 & J \\ 
          \hline 
          J^t & 0 
        \end{array} 
        \right],
$$
where $J$ is the $|X|\times |Y|$ all-one matrix.
We will show that
\begin{equation}\label{eq:incidence_eigenvalue}
        A^3 
        = (q-1)q^{k-3}M+q^{k-2} A.
\end{equation}
 By definition, for any $x\in X$ and $y\in Y$, $A^3(x,y)$ is the number of walks of length three of the form $xy'x'y$ in $\Gamma_{q,k-1}$. There are two cases.

\medskip

\noindent \textbf{Case 1: $xy\in E(\Gamma_{q,k-1})$.} 
When $x'=x$, the number of choices for $y'$ is $q^{k-2}$. When $x'\not =x$, the number of choices for $x'$ is $q-1$ and for each such $x'$, the number of choices for $y'$ is $q^{k-3}$. Thus in this case the number of walks $xy'x'y$ is $q^{k-2}+(q-1)q^{k-3}$.

\medskip

\noindent \textbf{Case 2: $xy\not\in E(\Gamma_{q,k-1})$.} 
Suppose $x=(x_0,x_1)$ and $x'=(x'_0,x'_1)$. If $x_0=x'_0$, then $x_1\not=x'_1$, and hence $x$ and $x'$ have no common neighbor. When $x_0\not =x'_0$, the number of choices for $x'$ is $q-1$ and for each such $x'$, the number of choices for $y'$ is $q^{k-3}$. Thus in this case the number of walks $xy'x'y$ is $(q-1)q^{k-3}$.

Combining the two cases above we  obtain \Cref{eq:incidence_eigenvalue}. Next, let $u_X$ be the characteristic vector of $X$, that is, $u_X(v)=1$ for each $v\in V(\Gamma_{q,k-1})$ and $u_X(v)=0$ otherwise. Similarly, let $u_Y$ be the characteristic vector of $Y$. Let $a_1=\sqrt{d_1}u_X+\sqrt{d_2}u_Y$ and let $a_n=\sqrt{d_1}u_X-\sqrt{d_2}u_Y$. It is easy to check that $\lambda_1=-\lambda_n=\sqrt{d_1d_2}$ and that $a_1$, $a_n$ are eigenvectors corresponding to $\lambda_1$ and $\lambda_n$. Since $M$ is symmetric, the spectral theorem implies that $M$ has an orthornormal basis of eigenvectors. Hence, for  each $1<i<n$, there exists an eigenvector $a_i$ corresponding to $\lambda_i$ such that $a_i$ is orthogonal to both $a_1$ and $a_n$. Thus $a_i$ is orthogonal to $u_X$ and $u_Y$, which implies that $M\cdot a_i=0$.  Multiplying both sides of \Cref{eq:incidence_eigenvalue} by $a_i$, we obtain $\lambda_i^3=q^{k-2}\lambda_i$. Because the rank of $A$ is larger than 2, there exists at least one $\lambda_i\not=0$, and hence $\lambda_i=\pm q^{\frac{k}{2}-1}$. Note that since $\Gamma_{q,k-1}$ is bipartite, we have $\lambda_i=\lambda_{n-i+1}$. Therefore, $\lambda_2=q^{\frac{k}{2}-1}$.
\end{proof}

Let $S$ be a subset of $X$ with size $|S|=rq$. If we pick $y\in Y$ uniformly at random, then the expectation of $|N(y)\cap S|$ is $r$. Thus intuitively, the vertex set of a ``typical'' clique in $H_{q,k}$ intersects $S$ in $\Theta(r)$ vertices. The following lemma shows that a substantial portion of all cliques are ``typical''. 

\begin{lem}\label{lemma:typical}
Let $S$ be a subset of $X$ with size $|S|=rq$. For $0<\delta<1$, let
$$
Y_{\delta}=\l\{y\in Y~\big|~ (1-\delta)r\le |N(y)\cap S|\le (1+\delta)r\r\}.
$$
Then $|Y_{\delta}|\ge \l(1-\frac{2}{\delta^2r}\r)q^{k-1}$.
\end{lem}

\begin{proof}
Let
$$
Y_+=\l\{y\in Y~\big|~ |N(y)\cap S|> (1+\delta)r\r\}~\text{and}~Y_-=\l\{y\in Y~\big|~ |N(y)\cap S|< (1-\delta)r\r\}.
$$
Apply \Cref{lemma:expander mixing} with $G=\Gamma_{q,k-1}$ and $T=Y_+$. Together with \Cref{lemma:incidence_eigenvalue}, we have
$$
|e(S,Y_+)-\frac{q}{q^2}rq|Y_+||\le q^{\frac{k}{2}-1}\sqrt{rq|Y_+|}.
$$
By definition, $e(S, Y_+)\ge |Y_+|(1+\delta)r$. Thus $\delta r|Y_+|\le q^{\frac{k-1}{2}}\sqrt{r|Y_+|}$, which implies $|Y_+|\le \frac{q^{k-1}}{\delta^2r}$. Similarly, we can show that $|Y_-|\le \frac{q^{k-1}}{\delta^2r}$. Therefore,
$$
|Y_\delta|=|Y|-|Y_+|-|Y_-|\ge \l(1-\frac{2}{\delta^2r}\r)q^{k-1}.
$$
\end{proof}

\section{Balanced Supersaturation}
In this section, we show that $H^{*}_{q,k}$ has balanced supersaturation with positive probability. We need to use the following inequality.
\begin{prop}[Hoeffding-Azuma Inequality~\cite{hoeffding1994probability,azuma1967weighted}]\label{prop:HAIneq}
Let $\lambda\ge0$ and $c_1, \dots, c_k>0$ be reals, and $Z=(Z_1,\dots, Z_k)$ be a martingale with $Z_0=\E[Z]$ and $|Z_i-Z_{i-1}|\le c_i$ for all $i\le k$. Then 
$$
\P[Z-Z_0<-\lambda]\le\exp\l(-\frac{\lambda^2}{2\sum_{i=1}^k c_i^2}\r).
$$
\end{prop}

Recall that $H^*_{q,k}$ is the $k$-uniform hypergraph obtained by replacing each maximal clique of $H_{q,k}$ with a random complete $k$-partite $k$-graph on the same vertex set. Concretely, for each $y\in Y$, we color the vertices in $N(y)$ with $k$ colors $\{1,\dots,k\}$ uniformly at random, and for each $1\le i\le k$ we let $X_{y,i}\subseteq N(y)$ be the set of vertices with color $i$, and then we replace the clique on $N(y)$ with a complete $k$-partite $k$-graph on $N(y)$ with $k$-partition $X_{y,1}\sqcup\dots\sqcup X_{y,k}$. Note that the colorings for different cliques are independent.

Given a $k$-graph $H$, let $\Delta_i(H)$ denote the maximum integer such that there exists $S\subseteq V(H)$ such that $|S|=i$ and the number of edges containing $S$ is $\Delta_i(H)$.

\begin{lem}\label{lemma:balanced_supsat}
For $q$ sufficiently large in terms of $k$, with positive probability, every $S\subseteq X$ with $|S|\ge 4kq$ satisfies the following. There exists a subgraph $H\subset H^*_{q,k}[S]$ such that, for all $1\le i\le k$,
$$
\Delta_i(H)\le \frac{6(16k)^{2k}|E(H)|}{|S|}\l(\frac{q^{\frac{1}{k-1}}}{|S|}\r)^{i-1}. 
$$
\end{lem}

\begin{proof}
For a fixed $S\subseteq X$ with $|S|\ge 4kq$, let $r=|S|/q\ge 4k\ge 12$ and let
$$
Y_{1/2}=\l\{y\in Y~\big|~ r/2\le |N_{\Gamma_{q,k-1}}(y)\cap S|\le 3r/2\r\}.
$$
By \Cref{lemma:typical} we have $|Y_{1/2}|\ge q^{k-1}/3$. 

Let $H$ be a subgraph of $H^*_{q,k}[S]$ with edge set
$$
E(H)=\l\{e\in E(H^*_{q,k}[S])~\big|~\exists y\in Y_{1/2}\text{~such that~}e\in N(y)\r\}.
$$
In other words, $H$ contains only edges that are in the ``typical'' cliques. Define the random variable $Z=|E(H)|$. For all $y\in Y_{1/2}$ and $v\in N_{\Gamma_{q,k-1}}(y)$, let $A_{y,v}$ be the random variable with values in $\{1,\dots,k\}$ such that $A_{y,v}=i$ if vertex $v$ receives color $i$ in the clique on $N_{\Gamma_{q,k-1}}(y)$. Let $B_1,\dots,B_t$ be an arbitrary order of $A_{y,v}$ for $y\in Y_{1/2}$ and $v\in N_{\Gamma_{q,k-1}}(v)$. Define a martingale $Z_0,\dots,Z_t$ where $Z_0=\E(Z)$ and $Z_i=\E(Z|B_1,\dots,B_i)$ for $1\le i\le t$. Observe that changing the color of a vertex $v$ in a typical clique will only affect the number of edges containing $v$ in that clique, which is at most $\binom{3r/2-1}{k-1}$, since a typical clique has size at most $3r/2$. Thus $|Z_i-Z_{i-1}|\le \binom{3r/2-1}{k-1}\le (2r)^{k-1}$.  Note that for any $k$ vertices in a typical clique, the probability that they form an edge in $H$ is $\frac{k\,!}{k^k}$. Hence, by linearity of expectation and the fact that a typical clique has size at least $r/2$ and that $r/2-k\ge r/4$ , we have
$$
\E(Z)\ge |Y_{1/2}|\binom{r/2}{k}\frac{k\,!}{k^k}\ge \frac{q^{k-1}}{3}\frac{(r/2-k)^k}{k^k} \ge \frac{r^kq^{k-1}}{3(4k)^k}.
$$
Thus by \Cref{prop:HAIneq} (Hoeffding-Azuma Inequality) with $\lambda=\frac{r^kq^{k-1}}{6(4k)^k}$ and $c_i=(2r)^{k-1}$,  we have
$$
\Pr\l(Z\le \frac{r^kq^{k-1}}{6(4k)^k}\r)\le \exp\l(-\frac{(\frac{r^kq^{k-1}}{6(4k)^k})^2}{2(3rq^{k-1}/2)((2r)^{k-1})^2}\r)\le\exp\l(-\frac{rq^{k-1}}{500(8k)^{2k}}\r).
$$
Using the union bound, the probability that there exists an $S\subseteq X$ with $|S|=s=rq\ge 4kq$ such that $Z\le \frac{r^kq^{k-1}}{6(4k)^k}$ is at most
$$
\sum_{s=4kq}^{q}\binom{q^2}{s}\exp\l(-\frac{sq^{k-2}}{500(8k)^{2k}}\r)\le \sum_{s=1}^{q}\exp\l(-\frac{sq^{k-2}}{1000(8k)^{2k}}\r)<1
$$
given that $q$ is sufficiently large in terms of $k$.

Hence with positive probability, for every $S\subseteq X$ with $|S|=rq\ge 4kq$, the corresponding $H$ satisfies $|E(H)|\ge\frac{r^kq^{k-1}}{6(8k)^{2k}}$. Let $J\subseteq S$ be such that $|J|=i$ and $1\le i\le r-1$. Note that the number of $y$ such that $J\subseteq N_{\Gamma_{q,k-1}}(y)$ is at most $q^{k-1-i}$, and for each such $y\in Y_{1/2}$ the number of edges in $N_{\Gamma_{q,k-1}}(y)\cap S$ containing $J$ is at most $\binom{3r/2-i}{k-i}\le (2r)^{k-i}$. Hence we have
$$
\Delta_i(H)\le (2r)^{k-i}q^{k-1-i}.
$$
In particular, we know that $\Delta_k(H)\le 1$. Combining the inequalities above we have for all $1\le i\le k$,
$$
\Delta_i(H)\le \frac{6(16k)^{2k}|E(H)|}{|S|}\l(\frac{q^{\frac{1}{k-1}}}{|S|}\r)^{i-1}. 
$$
\end{proof}

\section{Counting independent sets}
We make use of the hypergraph container method developed independently by Balogh, Morris and Samotij~\cite{balogh2015independent} and Saxton and Thomason~\cite{saxton2015hypergraph}. Here we make use of the following simplified version of  Theorem 1.5 in~\cite{morris2024asymmetric}:
\begin{thm}[Theorem 1.5, \cite{morris2024asymmetric}]\label{theorem:container}
For every integer $r \ge 2$, there exists a constant $\epsilon>0$ such that the following holds.
Let $B,L\ge 1$ be positive integers and let ${H}$ be a $k$-graph satisfying
    \begin{equation}\label{equation:container}
        \Delta_i({H})\leq \frac{|E(H)|}{L}\l(\frac{B}{|V(H)|}\r)^{i-1},~~\forall 1\le i\le  r.
    \end{equation}
    
Then there exists a collection $\mathcal{C}$ of subsets of $V(H)$ such that:
    \begin{enumerate}[align=left]
        \item[{\rm (i)}] For every independent set $I$ of $H$, there exists $C\in \mathcal{C}$ such that $I\subset C$;
        \item[{\rm (ii)}] For every $C\in\mathcal{C}$, $|C|\leq |V(H)|-\epsilon L$; 
        \item[{\rm (iii)}] We have \[|\mathcal{C}|\le \exp\l(\frac{\log\l(\frac{|V(H)|}{B}\r)B}{\epsilon}\r).\]
    \end{enumerate}
\end{thm}

Next, we use \Cref{theorem:container} together with \Cref{lemma:balanced_supsat} to count the number of independent sets of size $q^{\frac{1}{k-1}}(\log q)^2$ in $H^*_{q,k}$.

\begin{thm}\label{theorem:count_indset}
For every $k\ge 3$, there exists a constant $c>0$ such that, when $q$ is sufficiently large, we can fix an instance of $H^*_{q,k}$ such that the number of independent sets of size $t=q^{\frac{1}{k-1}}(\log q)^2$ of $H^*_{q,k}$ is at most
$$
\l(\frac{cq}{t}\r)^t.
$$
\end{thm}

\begin{proof}
By \Cref{lemma:balanced_supsat}, we can fix an instance of $H^*_{q,k}$ such that for every $S\subset V(H^*_{q,k})$ with $|S|\ge 4kq$ there exists a subgraph $H$ of $H^*_{q,k}[S]$ such that for all $1\le i\le k$,
\begin{equation}\label{equation:balanced}
\Delta_i(H)\le \frac{6(16k)^{2k}|E(H)|}{|S|}\l(\frac{q^{\frac{1}{k-1}}}{|S|}\r)^{i-1}.    
\end{equation}

We will first prove the following claim.
\begin{claim}\label{claim:onestep_container}
    There exists a constant $\epsilon>0$ such that for every $S\subset V(H^*_{q,k})$ with $|S|>4kq$, there exists a collection $\C$ of at most
    $$
    \exp\l(\frac{\log q\cdot q^{\frac{1}{k-1}}}{\epsilon}\r)
    $$
    subsets of $S$ such that:
    \begin{itemize}[align=left]
        \item[{\rm (i)}]
        For every independent set $I$ of $H^*_{q,k}[S]$, there exists $C\in\C$ such that $I\subset C$;
        \item[{\rm (ii)}] For every $C\in\C$, $|C|\le (1-\epsilon)|S|$.
    \end{itemize}
\end{claim}
\begin{proof}
Fix an arbitrary $S\subset V(H^*_{q,k})$ with $|S|\ge 4kq$. By \Cref{equation:balanced}, it is easy to check that equation~(\ref{equation:container}) holds for $H=H^*_{q,k}[S]$, $L=\frac{|S|}{6(16k)^{2k}}$ and $B=q^{\frac{1}{k-1}}$. Hence by Theorem~\ref{theorem:container}, there exist a constant $\epsilon$ (not depending on $S$) and a collection $\C$ of subsets of $S$ with the desired properties.
\end{proof}

Now we apply \Cref{claim:onestep_container} iteratively as follows. Fix the constant $\epsilon$ guaranteed by \Cref{claim:onestep_container}. Let $\C_0=\{V(H^*_{q,k})\}$. Let $t_0=|V(H^*_{q,k})|=q^2$ and let $t_i=(1-\epsilon)t_{i-1}$ for all $i\ge 1$. Let $m$ be the smallest integer such that $t_m\le 4kq$. Clearly $m=O(\log q)$. Given a set of containers $\C_i$ such that every $C\in\C_i$ satisfies $|C|\le t_i$, we construct $\C_{i+1}$ as follows: for every $C\in \C_i$, if $|C|\le t_{i+1}$, then we put it into $\C_{i+1}$; otherwise, if $|C|> t_{i+1}$, by \Cref{claim:onestep_container}, there exists a collection $\C'$ of containers for $H^*_{q,k}[C]$ such that every $C'\in\C'$ satisfies $|C'|<(1-\epsilon)|C|\le t_{i+1}$~--- now we put every element of $\C'$ into $\C_{i+1}$. Let $\C=\C_m$. Note that
$$
\frac{|\C_i|}{|\C_{i-1}|}\le \exp\l(\frac{\log q\cdot q^{\frac{1}{k-1}}}{\epsilon}\r).
$$
Thus
$$
|\C_m|=\prod^m_{i=1}\frac{|\C_i|}{|\C_{i-1}|}\le \exp\l(m\frac{\log q\cdot q^{\frac{1}{k-1}}}{\epsilon}\r).
$$
As $m=O(\log q)$, we conclude that
there exists a constant $c_1>0$ such that
$$
|\C|=|\C_m|\le \exp\l(c_1(\log q)^2q^{\frac{1}{k-1}}\r).
$$
Also, by definition, we have $|C|\le 4kq$ for every $C\in\C$.

Recall that $t=(\log q)^2 q^{\frac{1}{k-1}}$ and let $N_t$ be the number of independent set of $H$ of size $t$. Since every independent set of $H$ of size $t$ is contained in some $C\in\C$, we have, for some constant $c_2>0$,
$$
N_t\le |\C|\binom{4kq}{t}\le \l(\frac{c_2q}{t}\r)^t.
$$
\end{proof}

\section{Proof of Theorem~\ref{theorem:main}}

\begin{proof}[Proof of Theorem~\ref{theorem:main}]
For every sufficiently large prime power $q$, we let $t=(\log q)^2q^{\frac{1}{k-1}}$. By \Cref{theorem:count_indset} we can fix an instance of $H^*_{q,k}$ such that the number of independent sets of $H^*_{q,k}$ of size $t$ is at most
$$
\l(\frac{c_1q}{t}\r)^t.
$$ 
for some constant $c_1>0$. Let $W$ be a random subset of $V(H^*_{q,k})$ where each vertex is sampled independently with probability $p=\frac{t}{c_1q}$. Note that $p<1$ as $q$ is sufficiently large. Then the expected number of independent set of size $t$ in $H^*_{q,k}[W]$ is at most
$$
\l(\frac{c_1q}{t}\r)^tp^t\le 1.
$$

Let $W'\subseteq W$ be obtained by arbitrarily deleting one vertex in each independent set of size $t$. Thus the expectation of $|W'|$ is at least
$$
pq^2-1=\frac{(\log q)^2}{c_1}q^{\frac{k}{k-1}}-1.
$$

Hence there exists a choice $W'$ with at least this many vertices. Let $H'=H^*_{q,k}[W']$. By definition of $W'$, we have $\alpha(H')\le t$. Moreover, by \Cref{proposition:F-free} we know that $H'$ is $F$-free. Thus, we have
$$
r(F,t)\ge \frac{(\log q)^2}{c_1}q^{\frac{k}{k-1}}.
$$
Recall that $t=(\log q)^2q^{\frac{1}{k-1}}$. It is well-known that for every integer $n$ there exists a prime $q$ such that $n/2\le q\le n$. Thus for every $n$ sufficiently large, it is easy to find a prime $q$ such that
$$
(\log q)^2q^{\frac{1}{k-1}}\le n\le 2(\log q)^2q^{\frac{1}{k-1}}.
$$
Therefore we conclude that there exists a constant $c_2>0$ such that for all $n$ sufficiently large,
$$
r(F,n)\ge \frac{c_2n^k}{(\log n)^{2k-2}}.
$$
\end{proof}

\bibliographystyle{abbrv}
\bibliography{refs}

\begin{thebibliography}{10}

\bibitem{AKS1980}
M.~Ajtai, J.~Komlós, and E.~Szemerédi.
\newblock A note on {R}amsey numbers.
\newblock {\em Journal of Combinatorial Theory, Series A}, 29(3):354--360, 1980.

\bibitem{azuma1967weighted}
K.~Azuma.
\newblock Weighted sums of certain dependent random variables.
\newblock {\em Tohoku Mathematical Journal, Second Series}, 19(3):357--367, 1967.

\bibitem{balogh2015independent}
J.~Balogh, R.~Morris, and W.~Samotij.
\newblock Independent sets in hypergraphs.
\newblock {\em Journal of the American Mathematical Society}, 28(3):669--709, 2015.

\bibitem{bohman2010early}
T.~Bohman and P.~Keevash.
\newblock The early evolution of the {$H$}-free process.
\newblock {\em Inventiones mathematicae}, 181(2):291--336, 2010.

\bibitem{bohman2021dynamic}
T.~Bohman and P.~Keevash.
\newblock Dynamic concentration of the triangle-free process.
\newblock {\em Random Structures \& Algorithms}, 58(2):221--293, 2021.

\bibitem{bohman2016independent}
T.~Bohman, D.~Mubayi, and M.~Picollelli.
\newblock The independent neighborhoods process.
\newblock {\em Israel Journal of Mathematics}, 214:333--357, 2016.

\bibitem{CM}
J.~Cooper and D.~Mubayi.
\newblock Sparse hypergraphs with low independence number.
\newblock {\em Combinatorica}, 37(1):31--40, 2017.

\bibitem{EH72}
P.~Erd\H{o}s and A.~Hajnal.
\newblock On {R}amsey like theorems. {P}roblems and results.
\newblock In {\em Combinatorics ({P}roc. {C}onf. {C}ombinatorial {M}ath., {M}ath. {I}nst., {O}xford, 1972)}, pages 123--140. Inst. Math. Appl., Southend-on-Sea, 1972.

\bibitem{fiz2020triangle}
G.~Fiz~Pontiveros, S.~Griffiths, and R.~Morris.
\newblock The triangle-free process and the {R}amsey number {$R(3,k)$}.
\newblock {\em Mem. Amer. Math. Soc.}, 263(1274):v+125, 2020.

\bibitem{fox2021independent}
J.~Fox and X.~He.
\newblock Independent sets in hypergraphs with a forbidden link.
\newblock {\em Proceedings of the London Mathematical Society}, 123(4):384--409, 2021.

\bibitem{Hae95}
W.~H. Haemers.
\newblock Interlacing eigenvalues and graphs.
\newblock {\em Linear Algebra and its Applications}, 226/228:593--616, 1995.

\bibitem{hoeffding1994probability}
W.~Hoeffding.
\newblock Probability inequalities for sums of bounded random variables.
\newblock {\em The collected works of Wassily Hoeffding}, pages 409--426, 1994.

\bibitem{Kim1995}
J.~H. Kim.
\newblock The {R}amsey number $r(3, t)$ has order of magnitude $t^2/\log t$.
\newblock {\em Random Structures \& Algorithms}, 7(3):173--207, 1995.

\bibitem{kostochka2013hypergraph}
A.~Kostochka, D.~Mubayi, and J.~Verstra\"{e}te.
\newblock Hypergraph {R}amsey numbers: triangles versus cliques.
\newblock {\em Journal of Combinatorial Theory, Series A}, 120(7):1491--1507, 2013.

\bibitem{kostochka2014independent}
A.~Kostochka, D.~Mubayi, and J.~Verstra{\"e}te.
\newblock On independent sets in hypergraphs.
\newblock {\em Random Structures \& Algorithms}, 44(2):224--239, 2014.

\bibitem{li2001asymptotic}
Y.~Li, C.~C. Rousseau, and W.~Zang.
\newblock Asymptotic upper bounds for {R}amsey functions.
\newblock {\em Graphs and Combinatorics}, 17:123--128, 2001.

\bibitem{mattheus2024asymptotics}
S.~Mattheus and J.~Verstra\"{e}te.
\newblock The asymptotics of $r (4, t)$.
\newblock {\em Annals of Mathematics}, 199(2):919--941, 2024.

\bibitem{morris2024asymmetric}
R.~Morris, W.~Samotij, and D.~Saxton.
\newblock An asymmetric container lemma and the structure of graphs with no induced $4 $-cycle.
\newblock {\em Journal of the European Mathematical Society}, 2024.

\bibitem{MV2024}
D.~Mubayi and J.~Verstra\"ete.
\newblock A note on pseudorandom {R}amsey graphs.
\newblock {\em J. Eur. Math. Soc. (JEMS)}, 26(1):153--161, 2024.

\bibitem{phelps1986steiner}
K.~T. Phelps and V.~R{\"o}dl.
\newblock Steiner triple systems with minimum independence number.
\newblock {\em Ars Combin}, 21:167--172, 1986.

\bibitem{saxton2015hypergraph}
D.~Saxton and A.~Thomason.
\newblock Hypergraph containers.
\newblock {\em Inventiones mathematicae}, 201(3):925--992, 2015.

\bibitem{shearer1983note}
J.~B. Shearer.
\newblock A note on the independence number of triangle-free graphs.
\newblock {\em Discrete Mathematics}, 46(1):83--87, 1983.

\bibitem{spencer1975ramsey}
J.~Spencer.
\newblock {R}amsey's theorem—a new lower bound.
\newblock {\em Journal of Combinatorial Theory, Series A}, 18(1):108--115, 1975.

\bibitem{spencer1977asymptotic}
J.~Spencer.
\newblock Asymptotic lower bounds for {R}amsey functions.
\newblock {\em Discrete Mathematics}, 20:69--76, 1977.

\end{thebibliography}

\end{document}